\documentclass[10pt,reqno]{amsart}
\usepackage{pdfsync}
\usepackage{amssymb,amscd}
\usepackage{color}

\usepackage{graphicx}

\usepackage[bookmarks=true,hyperindex,pdftex,colorlinks,
citecolor=blue,urlcolor=blue]{hyperref}

\hypersetup{pdftitle={The version for compact operators of Lindenstrauss properties A and B},
pdfauthor={Miguel Martin}}

\theoremstyle{plain}
\newtheorem{theorem}{Theorem}[section]
\newtheorem{proposition}[theorem]{Proposition}
\newtheorem{prop}[theorem]{Proposition}
\newtheorem{corollary}[theorem]{Corollary}
\newtheorem{lemma}[theorem]{Lemma}
\newtheorem{fact}[theorem]{Fact}

\theoremstyle{definition}
\newtheorem{definition}[theorem]{Definition}

\newtheorem{question}[theorem]{Question}
\newtheorem{example}[theorem]{Example}

 \DeclareMathOperator{\Id}{\mathrm{Id}}
 \newcommand{\N}{\mathbb{N}}
  \newcommand{\C}{\mathbb{C}}
    \newcommand{\R}{\mathbb{R}}
  \newcommand{\D}{\mathbb{D}}
 \newcommand{\Ak}{A$^{\!k}$}
  
 \newcommand{\Bk}{B$^k$}

 \renewcommand{\geq}{\geqslant}
 \renewcommand{\leq}{\leqslant}

\begin{document}
\title{The version for compact operators of Lindenstrauss properties A and B}

\author[Miguel Martin]{}

  \thispagestyle{plain}

\thanks{Supported by Spanish MICINN and FEDER project no.\ MTM2012-31755, and by Junta de Andaluc\'{\i}a and FEDER grants FQM-185 and P09-FQM-4911.}


\date{}

\maketitle

\centerline{\textsc{\large Miguel Mart\'{\i}n}}

\begin{center}\small Departamento de An\'{a}lisis Matem\'{a}tico \\ Facultad de
Ciencias \\ Universidad de Granada \\ 18071 Granada, Spain \\
\emph{E-mail:} \texttt{mmartins@ugr.es} \\ \href{http://orcid.org/0000-0003-4502-798X}{ORCID: \texttt{0000-0003-4502-798X} }
\end{center}

\begin{abstract}
It has been very recently discovered that there are compact linear operators between Banach spaces which cannot be approximated by norm attaining operators. The aim of this expository paper is to give an overview of those examples and also of sufficient conditions ensuring that compact linear operators can be approximated by norm attaining operators. To do so, we introduce the analogues for compact operators of Lindenstrauss properties A and B.
\end{abstract}

\section{Introduction}
The study of norm attaining operators started with a celebrated paper by J.~Lindenstrauss of 1963 \cite{Lindens}. There, he provided examples of pairs of Banach spaces such that there are (bounded linear) operators between them which cannot be approximated by norm attaining operators. Also, sufficient conditions on the domain space or on the range space providing the density of norm attaining operators were given. We recall that an operator $T$ between two Banach spaces $X$ and $Y$ is said to \emph{attain its norm} whenever there is $x\in X$ with $\|x\|=1$ such that $\|T\|=\|T(x)\|$ (that is, the supremum defining the operator norm is actually a maximum).

Very recently, it has been shown that there exist compact linear operators between Banach spaces which cannot be approximated by norm attaining operators \cite{Martin-KNA}, solving a question that has remained open since the 1970s. We recall that an operator between Banach spaces is \emph{compact} if it carries bounded sets into relatively compact sets or, equivalently, if the closure of the image of the unit ball is compact. After the cited result of \cite{Martin-KNA}, it makes more sense to discuss sufficient conditions on the domain or on the range space ensuring density of norm attaining compact operators. The objective of this expository paper is to present the give a picture of old and new developments in this topic.

Let us first make some remarks about the existence of compact operators which do not attain their norm. It is clear that if the domain of a linear operator is a finite-dimensional space, then the image by it of the unit ball is actually compact and, therefore, the operator automatically attains its norm. This argument extends to infinite-dimensional reflexive spaces. Indeed, if $X$ is a reflexive space, every compact operator from $X$ into a Banach space $Y$ is completely continuous (i.e.\ it maps weakly convergent sequences into norm convergent sequences, see \cite[Problem 30 in p.~515]{Dun-Sch} for instance) and so the weak (sequential) compactness of the unit ball of $X$ gives easily the result. We refer to \cite[Theorem~6 in p.~16]{DiesGeom} for a discussion on when all bounded linear operators between reflexive spaces attain their norm. On the other hand, for every non-reflexive Banach space $X$ there is a continuous linear functional on $X$ which does not attain its norm (James' theorem, see \cite[Theorem~2 in p.~7]{DiesGeom} for instance). The multiplication of such functional by a fix non-zero vector of a Banach space $Y$ clearly produces a rank-one (hence compact) operator from $X$ into $Y$ which does not attain its norm. As a final remark, let us comment that for every infinite-dimensional Banach space $X$, there exists a bounded linear operator $T:X\longrightarrow c_0$ which does not attain its norm \cite[Lemma~2.2]{MarMerPay}.

Let us present a brief account on classical results about density of norm attaining operators. The expository paper \cite{Acosta-RACSAM} can be used for reference and background. We need some notation. Given two (real or complex) Banach spaces $X$ and $Y$, we write $L(X,Y)$ to denote the Banach space of all bounded linear operators from $X$ into $Y$, endowed with the operator norm. By $K(X,Y)$ and $F(X,Y)$ we denote the subspaces of $L(X,Y)$ of compact operators and finite-rank operators, respectively. We write $X^*$ for the (topological) dual of $X$, $B_X$ for the closed unit ball, and $S_X$ for the unit sphere. The set of norm attaining operators from $X$ into $Y$ is denoted by $NA(X,Y)$. The study on norm attaining operators started as a negative answer by J.~Lindenstrauss \cite{Lindens} to the question of whether it is possible to extend to the vector valued case the classical Bishop-Phelps theorem of 1961 \cite{BishopPhelps} stating that the set of norm attaining functionals is always dense in the dual of a Banach space. J.~Lindenstrauss introduced two properties to study norm attaining operators: a Banach space $X$ (resp.\ $Y$) has Lindenstrauss \emph{property A} (resp.\ \emph{property B}) if $NA(X,Z)$ is dense in $L(X,Z)$ (resp.\ $NA(Z,Y)$ is dense in $L(Z,Y)$) for every Banach space $Z$. It is shown in \cite{Lindens}, for instance, that $c_0$, $C[0,1]$ and $L_1[0,1]$ fail property $A$. Examples of spaces having property A (including reflexive spaces and $\ell_1$) and of spaces having property B (including $c_0$, $\ell_\infty$ and every finite-dimensional space whose unit ball is a polyhedron) are also shown in that paper. There are many extensions of Lindenstrauss results from which we will comment only a representative sample. With respect to property A, J.~Bourgain showed in 1977 that every Banach space with the Radon-Nikod\'{y}m property have property A and that, conversely, if a Banach space $X$ has property A in every equivalent norm, then it has the Radon-Nikod\'{y}m property (this direction needs a refinement due to R.~Huff, 1980). W.~Schachermayer (1983) and B.~Godun and S.~Troyanski (1993) showed that ``almost'' every Banach space can be equivalently renormed to have property A. With respect to property B, J.~Partington proved that every Banach space can be renormed to have property B (1982) and W.~Schachermayer showed that $C[0,1]$ fails the property (1983). W.~Gowers showed in 1990 that $\ell_p$ does not have property B for $1<p<\infty$, a result extended by M.~Acosta (1999) to all infinite-dimensional strictly convex Banach spaces and to infinite-dimensional $L_1(\mu)$ spaces. With respect to pairs of classical Banach spaces not covered by the above results, J.~Johnson and J.~Wolfe (1979) proved that, in the real case, $NA(C(K),C(S))$ is dense in $L(C(K),C(S))$ for all compact spaces $K$ and $S$, and C.~Finet and R.~Pay\'{a} (1998) showed the same result for the pair $(L_1[0,1],L_\infty[0,1])$. Concerning the study of norm attaining compact operators, J.~Diestel and J.~Uhl (1976) \cite{Diestel-Uhl-Rocky} showed that norm attaining finite-rank operators from $L_1(\mu)$ into any Banach space are dense in the space of all compact operators. This study was continued by J.~Johnson and J.~Wolfe \cite{JoWo} (1979), who proved the same result when the domain space is a $C(K)$ space or the range space is an $L_1$-space (only real case) or a predual of an $L_1$-space. In 2013, B.~Cascales, A.~Guirao, and V.~Kadets \cite[Theorem~3.6]{CGK} showed that for every uniform algebra (in particular, the disk algebra $A(\D)$), the set of norm attaining compact operators arriving to the algebra is dense in the set of all compact operators.

As we already mention, we would like to review here known results about density of norm attaining compact operators. As this question is too general, and imitating what Lindenstrauss did in 1963, we introduce the following two properties.

\begin{definition} Let $X$, $Y$ be Banach spaces.
\begin{enumerate}
\item[(a)] $X$ is said to satisfy \emph{property \Ak} if $K(X,Z)\cap NA(X,Z)$ is dense in $K(X,Z)$ for every Banach space $Z$.
\item[(b)] $Y$ is said to satisfy \emph{property \Bk} if $K(Z,Y)\cap NA(Z,Y)$ is dense in $K(Z,Y)$ for every Banach space $Z$.
\end{enumerate}
\end{definition}

Our first observation is that it is not clear whether, in general, property A implies property \Ak\, or property B implies property \Bk.

\begin{question}\label{question:A=>Ak}
Does Lindenstrauss property A imply property \Ak?
\end{question}

\begin{question}\label{question:B=>Bk}
Does Lindenstrauss property B imply property \Bk?
\end{question}

Nevertheless, all sufficient conditions for Lindenstrauss properties A and B listed above also implies, respectively, properties \Ak\ and \Bk. This is so because the usual way of establishing the density of norm attaining operators is by proving that every operator can be approximated by compact perturbations of it attaining the norm.

Besides of these examples, most of the known results which are specific for compact operators depend on some stronger forms of the approximation property. Let us recall the basic principles of this concept. We refer to the classical book \cite{Linden-Tz} for background and to \cite{Casazza} for a more updated account. A Banach space $X$ has the (Grothendieck) \emph{approximation property} if for every compact set $K$ and every $\varepsilon>0$, there is $R\in F(X,X)$ such that $\|x-R(x)\| <\varepsilon$ for all $x\in K$. Useful and classical results about the approximation theory are the following.

\begin{prop}\label{prop-AP}
Let $X$, $Y$ be Banach spaces.
\begin{enumerate}
\item[a)] $Y$ has the approximation property if and only if $\overline{F(Z,Y)}=K(Z,Y)$ for every Banach space $Z$.
\item[b)] $X^*$ has the approximation property if and only if $\overline{F(X,Z)}=K(X,Z)$ for every Banach space $Z$.
\item[c)] If $X^*$ has the approximation property, then so does $X$.
\item[d)] (Enflo) There exist Banach spaces without the approximation property. Actually, there are closed subspaces of $c_0$ failing the approximation property.
\end{enumerate}
\end{prop}

What is the relation between the approximation property and norm attaining compact operators? On the one hand, the negative examples given in \cite{Martin-KNA} exploit the failure of the approximation property of some subspaces of $c_0$ together with an easy extension of a geometrical property proved by Lindenstrauss for $c_0$ (see section~\ref{sec:negative} for details). On the other hand, most of the positive results for properties \Ak\, and \Bk\, which are not related to Lindenstrauss properties A and B use some strong form of the approximation property. Actually, all positive results in this line try to give a partial answer to one of the following two open questions.

\begin{question}\label{question-AP-Bk}
Does the approximation property imply property \Bk?
\end{question}

\begin{question}\label{question-dual-AP-Ak}
Does every Banach space whose dual has the approximation property have property \Ak\,?
\end{question}

By Proposition~\ref{prop-AP}.a, Question~\ref{question-AP-Bk} is equivalent to the following one, which is considered one of the most important open questions in the theory of norm attaining operators.

\begin{question}\label{question-fin-dim-B}
Does every finite-dimensional Banach space have Lindenstrauss property B?
\end{question}

Surprisingly, this question is open even for the two-dimensional Euclidean space.

\vspace*{1ex}

The outline of the paper is as follows. Section~\ref{sec:negative} contains a detailed explanation of the negative examples presented in \cite{Martin-KNA} and some new examples. We also use the techniques to present a couple of negative results on Lindenstrauss properties A and B.

In section~\ref{sec:domain} we collect all results about property \Ak. We start by listing the main known examples of Banach spaces with Lindenstrauss property A since all of them have property \Ak. Next, we show that an stronger version of the approximation property of the dual of a Banach space implies property \Ak\, \cite{JoWo}. Most of the examples of spaces with property \Ak\, given in the literature are actually proved in this way, as $L_1(\mu)$ spaces, $C_0(L)$ spaces. We are even able to produce some new examples as the preduals of $\ell_1$ and $M$-embedded Banach spaces with monotone basis.

The results for range spaces appear in section~\ref{sec:range}. Again, we start with a list of known examples of Banach spaces with Lindenstrauss property B which also have property \Bk. Next, we provide two partial positive answers to Question~\ref{question-AP-Bk} (i.e.\ whether the approximation property implies property \Bk). The first one is very simple: suppose that a Banach space $Y$ has the approximation property and every finite-dimensional subspace of it is contained in another subspace of $Y$ having property \Bk. Then, the whole space $Y$ has property \Bk. The second sufficient condition deals with the existence of a bounded net of projections converging to the identity in the strong operator topology and such that their ranges have property \Bk. These two ideas lead to a list including most of the known examples of Banach spaces with property \Bk: preduals of $L_1(\mu)$ spaces (in particular, $C_0(L)$ spaces), real $L_1(\mu)$ spaces, and polyhedral Banach spaces with the approximation property (in particular, subspaces of $c_0$ with the approximation property, both in the real and in the complex case). Besides these examples, uniform algebras have been very recently proved to have property \Bk\, \cite{CGK}.

We would like to finish this introduction with some comments about the complex case of Bishop-Phelps theorem and its relation to Question~\ref{question-fin-dim-B} for the two dimensional real Hilbert space. First, let us comment that there is a complex version of the Bishop-Phelps theorem, easily deductible from the real case, which states that for every complex Banach space $X$, complex-linear norm attaining functionals from $X$ into $\C$ are dense in $X^*=L(X,\C)$ (see \cite{Phelps1} or \cite[\S 2]{Phelps2}). But this does not imply that $\C$ viewed as the real two-dimensional Hilbert space have Lindenstrauss property B, as it does not allow to work with operators which are not complex-linear. On the other hand, V.~Lomonosov showed in 2000 \cite{Lomonosov} that there is a complex Banach space $X$ and a (non-complex symmetric) closed convex bounded subset $C$ of $X$ such that there is no element in $X^*$ attaining the supremum of its modulus on $C$.

\section{Negative examples}\label{sec:negative}
Our goal here is to present the recent results in \cite{Martin-KNA} providing examples of compact operators which cannot be approximated by norm attaining operators. The key idea is to combine the approximation property with the following simple geometric idea. Let $X$ and $Y$ be Banach spaces and let $T\in L(X,Y)$ with $\|T\|=1$. Suppose that $T$ attains its norm at a point $x_0\in S_X$ which is not an extreme point of $B_X$, let $z\in X$ be such that $\|x_0 \pm z\|\leq 1$ and observe that $\|Tx_0 \pm Tz\|\leq 1$. If $T x_0\in S_Y$ is an extreme point of $B_Y$, then $Tz=0$. Summarizing:
\begin{equation}\label{eq:xpmz}
\|x_0\pm z\| \leq 1 \ \ \text{and $Tx_0$ is an extreme point of $B_Y$} \qquad \Longrightarrow \qquad Tz=0.
\end{equation}
Two remarks are pertinent. First, the most number of vectors $z$'s we may use in the above equation, the most information we get about $T$. Second, to get that $Tx_0\in S_Y$ is an extreme point, the easiest way is to require that all points in $S_Y$ are extreme points of $B_Y$, that is, that $Y$ is \emph{strictly convex}.
\begin{figure}[h]
\centering
\includegraphics[width=10cm]{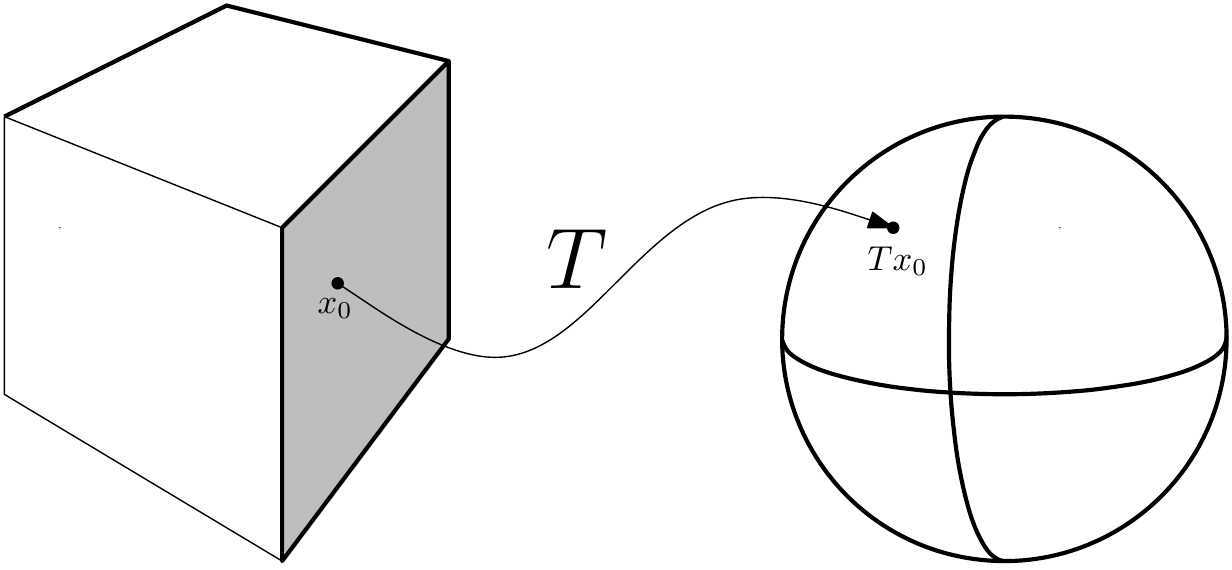}
\caption{An operator attaining its norm into a strictly convex space}
\label{figure}
\end{figure}
Observe that geometrically, this argument means that an operator into a strictly convex Banach space which attains its norm in the interior of a face of the unit sphere carries the whole face to the same point (see Figure~\ref{figure}).

Next, suppose that the set of $z$'s working in \eqref{eq:xpmz} generate a finite-codimensional subspace. Then $T$ has finite-rank. We have proved the key ingredient for all the examples.

\begin{lemma}[Geometrical key lemma, Lindenstrauss]\label{lemma:geometrical-key}
Let $X$, $Y$ be Banach spaces. Suppose that for every $x_0\in S_X$ the closed linear span of the set of those $z\in X$ such that $\|x_0\pm z\|\leq 1$ is finite-codimensional and that $Y$ is strictly convex. Then, $NA(X,Y)\subseteq F(X,Y)$.
\end{lemma}

This result was used by J.~Lindenstrauss \cite{Lindens} to give a direct proof of the fact that $c_0$ fails property A (this result is deductible from earlier parts in the same paper). How was this done? Just observing that for every $x_0\in S_{c_0}$, there is $N\in \N$ such that $\left\|x_0\pm \tfrac{1}{2}e_n\right\|\leq 1$ for every $n\geq N$ and that there is a strictly convex Banach space $Y$ isomorphic to $c_0$. The isomorphism from $c_0$ into $Y$ is not compact and so it cannot be approximated by norm attaining operators by Lemma~\ref{lemma:geometrical-key}. The argument also shows that strictly convex renorming of $c_0$ do not have property B. Similar arguments like the above ones, replacing $c_0$ by suitable Banach spaces, have been used by T.~Gowers \cite{Gowers} to show that $\ell_p$ fails Lindenstrauss property B, and by M.~Acosta \cite{Aco-Edinburgh,Aco-contemporary} to show that no infinite-dimensional strictly convex Banach space nor infinite-dimensional $L_1(\mu)$ space has Lindenstrauss property B.

But let us return to compact operators. Instead of looking for a non-compact operator which cannot be approximated by norm attaining operators using Lemma~\ref{lemma:geometrical-key}, what is done in \cite{Martin-KNA} is to use domain spaces without the approximation property. The key idea there is that Lindenstrauss' argument for $c_0$ extends to all of its closed subspaces.

\begin{lemma}[\textrm{\cite{Martin-KNA}}]\label{lemma-c0-noextreme}
Let $X$ be a closed subspace of $c_0$. Then, for every $x_0\in S_X$ the closed linear span of the set of those $z\in X$ such that $\|x_0\pm z\|\leq 1$ generates a finite-codimensional subspace.
\end{lemma}

\begin{proof}
Fix $x_0\in S_X$. As $x_0\in c_0$, there is $N\in \N$ such that $|x_0(n)|<1/2$ for every $n\geq N$. Now, consider the finite-codimensional subspace of $X$ given by
$$
Z=\bigl\{z\in X\,:\, z(i)=0\ \text{for } 1\leq i \leq N \bigr\},
$$
and observe that for every $z\in Z$ with $\|z\|\leq 1/2$, we have
$\left\|x_0 \pm z\right\|\leq 1$.
\end{proof}

Therefore, combining this result with Lemma~\ref{lemma:geometrical-key}, we get the following.

\begin{corollary}[\textrm{\cite{Martin-KNA}}]
\label{corollary:Xsubc0Ystrictlyconvex}
Let $X$ be a closed subspace of $c_0$ and let $Y$ be a strictly convex Banach space. Then, $NA(X,Y)\subseteq F(X,Y)$.
\end{corollary}

What is next? We just have to recall the approximation property and use Proposition~\ref{prop-AP}. Pick a closed subspace $X$ of $c_0$ without the approximation property. Then, $X^*$ also fails the approximation property, so there is a Banach space $Y$ and $T\in K(X,Y)$ which is not in the closure of $F(X,Y)$. Considering $Y$ as the range of $T$, which is separable, we may suppose that $Y$ is strictly convex (using an equivalent renorming, see \cite[\S II.2]{DGZ}). Now, Corollary~\ref{corollary:Xsubc0Ystrictlyconvex} gives that $T$ cannot be approximated by norm attaining operators. We have proved.

\begin{fact}[\textrm{\cite{Martin-KNA}}]
There exist compact operators between Banach spaces which cannot be approximated by norm attaining operators.
\end{fact}

Actually, we have proved the following result.

\begin{prop}[\textrm{\cite{Martin-KNA}}]\label{prop:c0noAP-failsAk}
Every closed subspace of $c_0$ whose dual does not have the approximation property fails property \Ak.
\end{prop}

An specially interesting example can be given using an space constructed by W.~Johnson and G.~Schechtman \cite[Corollary~JS, p.~127]{Joh-Oik} which is a closed subspace of $c_0$ with Schauder basis whose dual fails the approximation property.

\begin{example}[\textrm{\cite{Martin-KNA}}]\label{exa:JS}
There exist a subspace of $c_0$ with Schauder basis failing property \Ak.
\end{example}

Compare this result with Corollary~\ref{coro-subc0monotone} in next section.

Next we would like to produce more examples of spaces without property \Ak. We say that the norm of a Banach space $X$ \emph{locally depends upon finitely many coordinates} if for every $x\in X\setminus \{0\}$, there exist $\varepsilon>0$, a finite subset $\{f_1,f_2,\ldots,f_n\}$ of $X^*$ and a continuous function $\varphi:\R^n \longrightarrow \R$ such that $\|y\|=\varphi(f_1(y),f_2(y),\ldots,f_n(y))$ for every $y\in X$ such that $\|x-y\|<\varepsilon$. Clearly, this property is inherited by closed subspaces. We refer to \cite{Godefroy,Hajek-Zizler} and references therein for background. For instance, closed subspaces of $c_0$ have this property \cite[Proposition~III.3]{Godefroy}. Conversely, every infinite-dimensional Banach space whose norm locally depends upon finitely many coordinates contains an isomorphic copy of $c_0$ \cite[Corollary~IV.5]{Godefroy}. It is easy to extend the proof of Lemma~\ref{lemma-c0-noextreme} to this case and then use Lemma~\ref{lemma:geometrical-key} to get the following extension of Proposition~\ref{prop:c0noAP-failsAk}.

\begin{prop}\label{prop-LFC}
Let $X$ be a Banach space whose norm locally depends upon finitely many coordinates and whose dual fails the approximation property. Then $X$ does not have property \Ak.
\end{prop}

To prove the proposition, we need an extension of Corollary~\ref{corollary:Xsubc0Ystrictlyconvex}.

\begin{lemma}\label{lemma-LFC}
Let $X$ be a Banach space whose norm locally depends upon finitely many coordinates and let $Y$ be a strictly convex Banach space. Then, $NA(X,Y)\subseteq F(X,Y)$.
\end{lemma}

\begin{proof}
Fix $x_0\in S_X$. Consider $\varepsilon>0$, $\{f_1,f_2,\ldots,f_n\}\subset X^*$ and $\varphi:\R^n\longrightarrow \R$ given by the hypothesis for $x_0$. Let $Z=\bigcap_{i=1}^n \ker f_i$, which is finite-codimensional. For $z\in Z$ with $\|z\|<\varepsilon$, we have that
$$
\|x_0 \pm z\|=\varphi\bigl(f_1(x_0 \pm z),f_2(x_0 \pm z),\ldots,f_n(x_0 \pm z)\bigr)=\varphi\bigl(f_1(x_0),f_2(x_0),\ldots,f_n(x_0)\bigr)=1.
$$
Now, Lemma~\ref{lemma:geometrical-key} gives the result.
\end{proof}

Now, Proposition~\ref{prop-LFC} follows from the above lemma repeating the arguments used to prove Proposition~\ref{prop:c0noAP-failsAk}.

As another application of the above proposition, we get more examples of spaces failing property \Ak. We recall that a \emph{real} Banach space is said to be \emph{polyhedral} if the unit balls of all of its finite-dimensional subspaces are polyhedra (i.e.\ the convex hull of finitely many points). A typical example of polyhedral space is $c_0$ and hence, so are its closed subspaces. We refer to \cite{FonfLindPhelp} for background on polyhedral spaces. It is known that every polyhedral Banach space admits an equivalent norm which locally depends upon finitely many coordinates (see \cite[Theorem~3]{Hajek-Zizler}). This gives the following consequence of Proposition~\ref{prop-LFC}.

\begin{corollary}
Let $X$ be a real polyhedral Banach space such that $X^*$ does not have the approximation property. Then $X$ admits an equivalent norm failing property \Ak.
\end{corollary}

We would like next to use Lemma~\ref{lemma-LFC} to provide an alternative proof of the result by T.~Gowers \cite{Gowers} stating  that infinite-dimensional $L_p$ spaces fails Lindenstrauss property B.

\begin{corollary}[Alternative proof of a result of Gowers]
For $1<p<\infty$, the infinite dimensional $L_p$-spaces fail Lindenstrauss property B.
\end{corollary}

\begin{proof}
Fix $1<p<\infty$ and let suppose that $Y=L_p(\mu)$ is infinite-dimensional. By \cite[Theorem~1.1]{Gasparis-MathAnn}, there is a Banach space $E_p$ which is polyhedral (up to a renorming) and contains a quotient isomorphic to $\ell_p$. Let $X$ be an equivalent renorming of $E_p$ whose norm locally depends upon finitely many coordinates (see \cite[Theorem~3]{Hajek-Zizler}). Then, there is a non-compact bounded linear operator from $X$ into $Y$. This operator cannot be approximated by norm attaining operators by Lemma~\ref{lemma-LFC}.
\end{proof}

We now deal with the range space. Let $Y$ be a strictly convex Banach space without the approximation property. By a result of A.~Grothendieck \cite[Theorem~18.3.2]{Jarchow}, there is a closed subspace $X$ of $c_0$ such that $F(X,Y)$ is not dense in $K(X,Y)$. This, together with Corollary~\ref{corollary:Xsubc0Ystrictlyconvex}, show that $Y$ fails property \Bk.

\begin{prop}[\textrm{\cite{Martin-KNA}}]
Every strictly convex Banach space without the approximation property fails property \Bk.
\end{prop}

The same kind of arguments can be applied to subspaces of complex $L_1(\mu)$ spaces without the approximation property. Indeed, every subspace $Y$ of the complex $L_1(\mu)$ space is \emph{complex strictly convex} (see \cite[Proposition~3.2.3]{Istra}) and this means that for every $y\in Y$ with $\|y\|=1$ and $z\in Y$, the condition $\|y + \theta z\|\leq 1$ for every $\theta\in \C$ with $|\theta|=1$ implies $z=0$. An obvious adaption of the proofs of Corollary~\ref{corollary:Xsubc0Ystrictlyconvex} and of the above proposition, provide the following result.

\begin{prop}[\textrm{\cite{Martin-KNA}}]
Every closed subspace of the complex space $L_1(\mu)$ without the approximation property fails property \Bk.
\end{prop}

It is even possible to produce a Banach space $Z$ and a compact endomorphism of $Z$ which cannot be approximated by norm attaining operators.

\begin{example}[\textrm{\cite{Martin-KNA}}]
There exists a Banach space $Z$ and a compact operator from $Z$ into itself which cannot be approximated by norm attaining operators.
\end{example}

This is an immediate consequence of the following lemma, which is proved in \cite[Theorem~8]{Martin-KNA}.

\begin{lemma}[\textrm{\cite{Martin-KNA}}]
Let $X$, $Y$ be Banach spaces and let $Z=X\oplus_\infty Y$. If $NA(Z,Z)\cap K(Z,Z)$ is dense in $K(Z,Z)$ then $NA(X,Y)\cap K(X,Y)$ is dense in $K(X,Y)$.
\end{lemma}

We finish the section about negative examples with a couple of easy consequences of Lemma~\ref{lemma-LFC} about Lindenstrauss property A.

\begin{prop}
Let $X$ be a infinite-dimensional Banach space whose norm depends upon locally many coordinated. Then $X$ does not have Lindenstrauss property A.
\end{prop}

\begin{proof}
Let $Y$ be a strictly convex renorming of $X$. By Lemma~\ref{lemma-LFC}, $NA(X,Y)\subset F(X,Y)$. As $X$ is infinite-dimensional, the isomorphism from $X$ into $Y$ is non-compact and, therefore, it cannot be approximated by norm attaining operators.
\end{proof}

In particular, we get the result for subspaces of $c_0$.

\begin{corollary}
No infinite-dimensional closed subspace of $c_0$ satisfies Lindenstrauss property A.
\end{corollary}

Let us observe that this result solves in the negative Question~13 of \cite{Martin-KNA} as it is written there. As we have done here, it can be solved using arguments from that paper. But, actually, there is an erratum in the statement of this question and the exact question that the author wanted to propose is about property \Ak\, (see Question~\ref{question:c0-metric}).

\section{Positive results on domain spaces}\label{sec:domain}

As we commented in the introduction, every compact operator whose domain is reflexive attains its norm. In particular, reflexive spaces have property \Ak. To get more examples, we first recall that even it is not known whether Lindenstrauss property A implies property \Ak (Question~\ref{question:A=>Ak}), the usual way to prove property A for a Banach space $X$ is by showing that every operator from $X$ can be approximated by compact perturbations of it attaining the norm. Therefore, the known examples of spaces with property A actually have property \Ak. The main examples of this kind are spaces with the Radon-Nikod\'{y}m property (J.~Bourgain 1977 \cite{Bourgain}) and those with property $\alpha$ (W.~Schachermayer 1983 \cite{Schacher-alpha}). Let us start with the Radon-Nikod\'{y}m property, which does not need presentation as it is one of the classical properties studied in geometry of Banach spaces. Let us just recall that reflexive spaces and $\ell_1$ have it. Bourgain's result is much deeper than what we are going to present here and the paper also contains a kind of converse result. The paper \cite{Bourgain} is consider one of the cornerstones in the theory of norm attaining operators and connects this theory with the geometric concept of dentability. The proof given by Bourgain is based on a variational principle introduced in the same paper and approximates every operator by nuclear (hence compact) perturbations of it attaining the norm. The variational principle  was extended to the non-linear case in 1978 by C.~Stegall \cite{Stegall1978} getting rank-one perturbations. We are not going to present the proof here; we refer the reader to the 1986 paper \cite{Stegall1986} of C.~Stegall for a simpler proof and applications.

\begin{theorem}[Bourgain]\label{teo-Bourgain}
The Radon-Nikod\'{y}m property implies property \Ak.
\end{theorem}

Let us introduce the definition of property $\alpha$, probably less known. A Banach space $X$ has \emph{property $\alpha$} if there are two sets $\{ x_i\,:\, i \in I\} \subset S_X$, $\{ x^{*}_i\,:\, i \in I \} \subset S_{X^*}$ and a constant $0 \leq \rho <1$ such that the following conditions hold:
\begin{enumerate}
\item[$(i)$] $x^{*}_i(x_i)=1$, $\forall i \in I $.
\item[$(ii)$] $|x^*_i(x_j)|\leq \rho <1$ if $i, j \in I, i \ne j$.
\item[$(iii)$] $B_X$ is the absolutely closed convex hull of $\{x_i\,:\, i\in I\}$.
\end{enumerate}
This property was introduced by W.~Schachermayer \cite{Schacher-alpha} as a strengthening of a property used by  J.~Lindenstrauss in the seminal paper \cite{Lindens}. We refer to \cite{Moreno} and references therein for more information and background. The prototype of Banach space with property $\alpha$ is $\ell_1$.

\begin{prop}[Schachermayer]\label{prop-alpha}
Property $\alpha$ implies property \Ak.
\end{prop}

\begin{proof}
Let $X$ be a Banach space having property $\alpha$ with constant $\rho\in [0,1)$, and let $Y$ be a Banach space. Fix $T\in K(X,Y)$, $T\neq 0$, and $\varepsilon>0$. We find $i\in I$ such that
$$
\|T x_{i}\|> \frac{\|T\|(1+\varepsilon\rho)}{1+\varepsilon}
$$
and define $S\in K(X,Y)$ by
$$
S x = T x + \varepsilon x_{i}^*(x)Tx_{i} \qquad (x\in X).
$$
Then $\|S x_{i}\|> \|T\|(1+\varepsilon\rho)$, while $\|S x_j\|\leq \|T\|(1 + \varepsilon \rho)$ for every $j\neq i$. This gives that $S\in NA(X,Y)$ and it is clear that $\|T-S\|\leq \varepsilon\|T\|$.
\end{proof}

The main utility of property $\alpha$ is that many Banach spaces can be renormed with property $\alpha$ (B.~Godun and S.~Troyanski 1983 \cite{Godun-Troyanski}, previous results by W.~Schachermayer \cite{Schacher-alpha}), so we obtain that property \Ak\, has no isomorphic consequences in most cases.

\begin{corollary}
Every Banach space $X$ with a biorthogonal system whose cardinality is equal to the density character of $X$ can be equivalently renormed to have property \Ak. In particular, this happens if $X$ is separable.
\end{corollary}

Let us comment that property \Ak\ for a Banach space $X$ does not imply that for every Banach space $Y$, norm attaining finite-rank operators from $X$ into $Y$ are dense in $K(X,Y)$. Indeed, by the above, all reflexive spaces have property \Ak\,, while there are reflexive spaces whose duals fail the approximation property (even subspaces of $\ell_p$ for $p\neq 2$).

Let us pass to discuss on results which are specific of property \Ak\, and do not follow from property A. All results we know of this kind follow from the same general principle: an stronger version of the approximation property of the dual, and so give partial answers to Question~\ref{question-dual-AP-Ak}. The argument appeared in the 1979 paper by J.~Johnson and J.~Wolfe \cite[Lemma~3.1]{JoWo}. An (easy) proof can be found in \cite[Proposition~11]{Martin-KNA}.

\begin{proposition}[Johnson-Wolfe]\label{prop-suficiente}
Let $X$ be a Banach space. Suppose there is a net $(P_\alpha)$ of finite-rank contractive projections on $X$ such that $(P_\alpha^*)$ converges to $\Id$ in the strong operator topology (i.e.\ for every $x^*\in X^*$, $(P_\alpha^* x^*)\longrightarrow x^*$ in norm). Then $X$ has property \Ak.
\end{proposition}

\begin{proof}[Sketch of the proof]
Let $Y$ be a Banach space and consider $T\in K(X,Y)$. First, all operators $T P_\alpha$ attain their norm since $TP_\alpha(B_X)=T(B_{P_\alpha(X)})$ and $B_{P_\alpha(X)}$ is compact. By using the compactness of $T$, it can be easily proved that $(P_\alpha^*T^*)\longrightarrow T^*$, so  $(TP_\alpha)\longrightarrow T$.
\end{proof}

This result was used in the cited paper \cite{JoWo} to get examples of spaces with property \Ak. In \cite[Proposition~3.2]{JoWo} it is shown that real $C(K)$ spaces have the property given in Proposition~\ref{prop-suficiente}, but the proof easyly extends to real or complex $C_0(L)$ spaces.

\begin{example}[Johnson-Wolfe]\label{exa:AkC_0}
For every locally compact space Hausdorff space $L$, the space $C_0(L)$ has property \Ak.
\end{example}

In 1976, J.~Diestel and J.~Uhl \cite{Diestel-Uhl-Rocky} showed that $L_1(\mu)$ spaces have property \Ak.

\begin{example}[Diestel-Uhl]\label{exa:AkL_1}
For every positive measure $\mu$, the space $L_1(\mu)$ has property \Ak.
\end{example}

If the measure is finite, the above result also follows from Proposition~\ref{prop-suficiente}. The general case can be obtained from the finite measure case using the following lemma, which is just the immediate adaptation to compact operators of \cite[Lemma~2]{PayaSaleh}.

\begin{lemma} \label{prop-sumas}
Let $\{X_i\,:\,i\in I\}$ be a non-empty family of Banach spaces and let $X$ denote the $\ell_1$-sum of the family. Then $X$ has property \Ak\ if and only if $X_i$ does for every $i\in I$.
\end{lemma}

Compare Examples \ref{exa:AkC_0} and \ref{exa:AkL_1} with the result by W.~Schachermayer that $NA(L_1[0,1],C[0,1])$ is not dense in $L(L_1[0,1],C[0,1])$ \cite{Schachermayer-classical}.

More examples of spaces with property \Ak\ can be derived from Proposition~\ref{prop-suficiente}. We start with a new result about preduals of $\ell_1$.

\begin{corollary}
Let $X$ be a Banach space such that $X^*$ is isometrically isomorphic to $\ell_1$. Then $X$ has property \Ak.
\end{corollary}

\begin{proof}
Let $(x_n^*)_{n\in \N}$ be a Schauder basis of $X^*$ isometrically equivalent to the usual $\ell_1$-basis and for every $n\in \N$, let $Y_n$ the linear span of $\{x_1^*,\ldots,x_n^*\}$. In the proof of \cite[Corollary~4.1]{Gasparis}, a sequence of $w^*$-continuous contractive projections $Q_n:X^*\longrightarrow X^*$ with $Q_n(X^*)=Y_n$ is constructed. The $w^*$-continuity of $Q_n$ provides us with a sequence of finite-rank contractive projections on $X$ satisfying the hypothesis of Proposition~\ref{prop-suficiente}. Let us note that the paper \cite{Gasparis} deals only with real spaces, but the proofs of the particular results we are using here work in the complex case as well.
\end{proof}

We do not know whether the corollary above extends to isometric preduals of arbitrary $L_1(\mu)$ spaces.

\begin{question}
Do all preduals of $L_1(\mu)$ spaces have property \Ak?
\end{question}

Proposition~\ref{prop-suficiente} clearly applied to spaces with a shrinking monotone Schauder basis. Recall that a Schauder basis of a Banach space $X$ is said to be \emph{shrinking} if its sequence of coordinate functionals is a Schauder basis of $X^*$.

\begin{corollary}
Every Banach space with a shrinking monotone Schauder basis has property \Ak.
\end{corollary}

It is well-known that an unconditional Schauder basis of a Banach space is shrinking if the space does not contain $\ell_1$ (see \cite[Theorem 3.3.1]{Albiac-Kalton} for instance), so the following particular case appears.

\begin{corollary}
Let $X$ be a Banach space with unconditional monotone Schauder basis which does not contain $\ell_1$. Then $X$ has property \Ak.
\end{corollary}

For the class of $M$-embedded spaces, this last result can be improved removing the unconditionality condition on the basis, by using the 1988 result of G.~Godefroy and P.~Saphar that Schauder bases in $M$-embedded spaces with basis constant less than $2$ are shrinking (see \cite[Corollary~III.3.10]{HWW}, for instance). We recall that a Banach space $X$ is said to be \emph{$M$-embedded} if $X^\perp$ is the kernel of an $L_1$-projection in $X^*$ (i.e.\ $X^*= X^\perp \oplus Z$ for some $Z$ and $\|x^\perp + z\|=\|x^\perp\|+\|z\|$ for every $x^\perp\in X^\perp$ and $z\in Z$). We refer the reader to \cite{HWW} for background.

\begin{corollary}
Every $M$-embedded space with monote Schauder basis has property \Ak.
\end{corollary}

As $c_0$ is an $M$-embedded space \cite[Examples~III.1.4]{HWW} and $M$-embeddedness passes to closed subspaces \cite[Theorem~III.1.6]{HWW}, we get the following interesting particular case

\begin{corollary}[\textrm{\cite[Corollary~12]{Martin-KNA}}]
\label{coro-subc0monotone}
Every closed subspace of $c_0$ with monotone Schauder basis has property \Ak.
\end{corollary}

Compare this result with the example of a closed subspace of $c_0$ with Schauder basis failing property \Ak\ (Example~\ref{exa:JS}). It is an interesting question whether Corollary~\ref{coro-subc0monotone} extends to every closed subspace of $c_0$ with the metric approximation property.

\begin{question}\label{question:c0-metric}
Does every closed subspace of $c_0$ with the metric approximation property have property \Ak?
\end{question}

We finish this section which an easy observation about Questions \ref{question:A=>Ak} and \ref{question-dual-AP-Ak}. Even though we do not know whether either property A or the approximation property of the dual is sufficient to get property \Ak, we may provide the following partial answer.

\begin{prop}
Let $X$ be a Banach space having Lindenstrauss property A and such that $X^*$ has the approximation property. Then, $X$ has property \Ak.
\end{prop}

\begin{proof}
As $X^*$ has the approximation property, it is enough to show that every finite-rank operator from $X$ can be approximated by finite-rank norm attaining operators (see Proposition~\ref{prop-AP}.b). Fix a finite-rank operator $T:X\longrightarrow Y$ and write $Z=T(X)$ which is finite-dimensional. As $X$ has property A, we have that $NA(X,Z)=NA(X,Z)\cap K(X,Z)$ is dense in $L(X,Z)=K(X,Z)$, so we may find a sequence $T_n\in NA(X,Z)$ converging to $T$ (viewed as an operator from $X$ into $Z$). It is now enough to consider the operators $T_n$ as elements of $NA(X,Y)\cap K(X,Y)$.
\end{proof}

\section{Positive results on range spaces}\label{sec:range}

The first example of a Banach space with property \Bk\, is the base field by the classical Bishop-Phelps theorem \cite{BishopPhelps}. This implies that every rank-one operator can be approximated by norm attaining rank-one operators. It is not known whether this extends in general to finite-rank operators (Question~\ref{question-fin-dim-B}) or, even, to rank-two operators.

To get more positive examples, and analogously to what is done in the previous section, we start by recalling that it is not known whether Lindenstrauss property B implies property \Bk\, (Question~\ref{question:B=>Bk}) but, nevertheless, the usual way to prove property B for a Banach space $X$ is by showing that every operator into $Y$ can be approximated by compact perturbations of it attaining the norm. This is what happens with property $\beta$, which is the main example of this kind. A Banach space $Y$ has \emph{property $\beta$} (J.~Lindenstrauss 1963 \cite{Lindens}) if there are two sets $\{ y_i\,:\, i \in I\} \subset S_Y$, $\{ y^{*}_i\,:\, i \in I \} \subset S_{Y^*}$ and a constant $0 \leq \rho <1$ such that the following conditions hold:
\begin{enumerate}
\item[$(i)$] $y^{*}_i(y_i)=1$, $\forall i \in I $.
\item[$(ii)$] $|y^*_i(y_j)|\leq \rho <1$ if $i, j \in I, i \ne j$.
\item[$(iii)$] For every $y\in Y$,  $\|y\|=\sup\nolimits_{i \in I}  \bigl|y^{*}_i(y)\bigr|$.
\end{enumerate}
We refer to \cite{Moreno} and references therein for more information and background. Examples of Banach spaces with property $\beta$ are closed subspace of $\ell_\infty(I)$ containing the canonical copy of $c_0(I)$ and real finite-dimensional Banach spaces whose unit ball is a polyhedron (actually, these are the only real finite-dimensional spaces with property $\beta$).

\begin{prop}[Lindenstrauss]\label{prop-beta}
Property $\beta$ implies property \Bk.
\end{prop}

\begin{proof}
Let $Y$ be a Banach space with property $\beta$ with constant $\rho\in [0,1)$ and let $X$ be a Banach space. Fix $T\in K(X,Y)$ with $\|T\|=1$ and $\varepsilon>0$, and consider $0<\delta<\varepsilon/2$ such that $$
\left(1 + \frac{\varepsilon}{2}\right)(1-\delta)>1 + \rho\left(\frac{\varepsilon}{2}+\delta\right).
$$
Next, we find $i\in I$ such that
$$
\|T^* y^*_{i}\|>1-\delta
$$
and apply Bishop-Phelps theorem \cite{BishopPhelps} to get  $x^*_0\in X^*$ attaining its norm such that
$$
\|x_0^*\|=\|T^* y_i^*\| \quad \text{and} \quad \|T^* y^*_{i} - x_0^*\|<\delta.
$$
Define $S\in K(X,Y)$ by
$$
S x = T x + \left[\left(1+\frac{\varepsilon}{2}\right) x_0^*(x) - y^*_i(Tx)\right] y_i \qquad (x\in X).
$$
Then, it is immediate to check that $\|T-S\|<\frac{\varepsilon}{2}+\delta <\varepsilon$. Now, we have that
$$
S^*y_i^*=\left(1+\frac{\varepsilon}{2}\right) x_0^*, \quad \|S^* y_i^*\|\geq \left(1+\frac{\varepsilon}{2}\right) (1-\delta)
$$
and, for $j\neq i$, we have
$$
\|S^* y^*_j\|\leq 1 + \rho \left(\frac{\varepsilon}{2}+\delta\right).
$$
This shows that $S^*$ attains its norm at $y_i^*\in S_{Y^*}$. As $S^* y_i^* = \left(1+\tfrac{\varepsilon}{2}\right)x_0^*\in X^*$ also attains its norm, it follows that $S\in NA(X,Y)$.
\end{proof}

R.~Partington proved in 1982 \cite{Partington} that every Banach space can be renormed with property $\beta$, so we obtain that property \Bk\, does not have isomorphic consequences.

\begin{corollary}
Every Banach space can be equivalently renormed to have property \Bk.
\end{corollary}

Let us comment that property \Bk\, for a Banach space $Y$ does not imply that for every Banach space $X$, norm attaining finite-rank operators from $X$ into $Y$ are dense in $K(X,Y)$. Indeed, by the above, there are many Banach spaces with property \Bk\, failing the approximation property.

Let us comment that the knowledge about property B is more unsatisfactory than the one about property A. Besides of property $\beta$, the only sufficient condition we know for property B is the so-called property quasi-$\beta$, introduced in 1996 by M.~Acosta, F.~Agirre and R.~Pay\'{a} \cite{AcoAguipaya}. We are not going to give the definition of this property here (it is a weakening of property $\beta$), but let us just comment that there are even examples of finite-dimensional spaces with property quasi-$\beta$ which do not have property $\beta$. Again, the proof of the fact that property quasi-$\beta$ implies property B can be adapted to the compact case.

\begin{prop}[Acosta-Aguirre-Pay\'{a}]
Property quasi-$\beta$ implies property \Bk.
\end{prop}

Let us pass to discuss on results which are specific of property \Bk\, and which are not related to property B. Most of the results that we know of this kind follow from two general principles. The first one is the following straightforward proposition.

\begin{prop}\label{prop:general-principle-Bk-polyhedral}
Let $Y$ be a Banach space with the approximation property. Suppose that for every finite-dimensional subspace $W$ of $Y$, there is a closed subspace $Z$ having Lindenstrauss property B such that $W\leq Z \leq Y$. Then $Y$ has property \Bk.
\end{prop}

This result applied to those Banach spaces with the approximation property satisfying that all of its finite-dimensional subspaces have Lindenstrauss property B. This is the case of the polyhedral spaces, as finite-dimensional polyhedral spaces clearly fulfil property $\beta$ \cite{Lindens}.

\begin{corollary}
A polyhedral Banach space with the approximation property has property \Bk.
\end{corollary}

To deal with the complex case, we observe that polyhedrality is equivalent to the fact that the norm of each finite-dimensional subspace can be calculated as the maximum of the absolute value of finitely many functionals, and this implies property $\beta$ also in the complex case. With this idea, the result above can be extended to the complex case.

\begin{proposition}
Let $Y$ be a complex Banach space with the approximation property such that for every finite-dimensional subspace, the norm of the subspace can be calculated as the maximum of the modulus of finitely many functionals. Then $Y$ has property \Bk.
\end{proposition}

It is easy to see that closed subspaces of $c_0$ satisfy this condition (see \cite[Proposition~III.3]{Godefroy}).

\begin{example}
Closed subspaces of real or complex $c_0$ with the approximation property have property \Bk.
\end{example}

The main limitation of Proposition~\ref{prop:general-principle-Bk-polyhedral} is that we only know few examples of finite-dimensional Banach spaces with property B. If we use property $\beta$ (equivalent here to polyhedrality), what we actually are requiring in that proposition is that all finite-dimensional subspaces have property B. To get more examples, we present the second general principle to get property \Bk, which appeared in \cite[Lemma~3.4]{JoWo}.

\begin{prop}[Johnson-Wolfe]\label{prop:suf-Bk-projections}
A Banach space $Y$ has property \Bk\, provided that there is a net of projections $\{Q_\lambda\}$ in $Y$, with $\sup_\lambda \|Q_\lambda\|<\infty$, converging to $\Id_Y$ in the strong operator topology (i.e.\ $Q_\lambda(y)\longrightarrow y$ in norm for every $y\in Y$), and such that $Q_\lambda(Y)$ has property \Bk\, for every $\lambda$.
\end{prop}

\begin{proof}[Sketch of the proof]
Let $X$ be a Banach space and fix $T\in K(X,Y)$. The compactness of $T$ allows to show that $Q_\lambda T$ converges in norm to $T$. Therefore, it is enough to prove that each $Q_\lambda T$ can be approximated by norm attaining compact operators, but this is immediate since $Q_\lambda T$ arrives to the space $Q_\lambda(Y)$ which has property \Bk.
\end{proof}

In the above proposition, if we require the projections to have finite-dimensional range, the condition is an stronger form of the approximation property.

This result was used in \cite{JoWo} to show that (real or complex) isometric preduals of $L_1(\mu)$ spaces have property \Bk. Indeed, by a classical result of A.~Lazar and J.~Lindenstrauss, the projections $Q_\lambda$ can be chosen to have $\|Q_\lambda\|=1$ and $Q_\lambda(Y)\equiv \ell_\infty^{n_\lambda}$ (see \cite[Chapter 7]{Lacey} for instance).

\begin{corollary}[Johnson-Wolfe]
Every predual of a real or complex $L_1(\mu)$ space has property \Bk.
\end{corollary}

In particular,

\begin{example}[Johnson-Wolfe]
Real or complex $C_0(L)$ spaces have property \Bk.
\end{example}

Proposition~\ref{prop:suf-Bk-projections} also applied to \emph{real} $L_1(\mu)$ spaces. Indeed, it is enough to consider conditional expectations to finite collections of subsets of positive and finite measure, and use the density of simple functions in $L_1(\mu)$. Doing that, we get a net $(Q_\lambda)$ of norm-one projections converging to the identity in the strong operator topology such that $Q_\lambda(L_1(\mu))\equiv \ell_1^{n_\lambda}$. In the real case, $\ell_1^{m}$ is polyhedral, so it has property $\beta$; in the complex case, $\ell_1^m$ does not have property $\beta$ and it is not known whether it has property B.

\begin{example}[Johnson-Wolfe]
For every positive measure $\mu$, the \emph{real} space $L_1(\mu)$ has property \Bk.
\end{example}

The last familly of spaces with property \Bk\, that we would like to present here is the one of uniform algebras. The result recently appeared in \cite[R2 in p.~380]{CGK} and the proof is completely different from the previous ones in this chapter, as the authors do not use any kind of approximation property, but a nice complex version of Urysohn lemma constructed in the same paper. We recall that a \emph{uniform algebra} is a closed subalgebra of a complex $C(K)$ space that separates the points of $K$.

\begin{prop}[Cascales-Guirao-Kadets]
Every uniform algebra has property \Bk.
\end{prop}

\vspace*{0.2cm}

\noindent \textbf{Acknowledgment:\ } The author is grateful to Rafael Pay\'{a} for many fruitful conversations about the content of this paper. He also thanks Mar\'{\i}a Acosta, Joe Diestel, Ioannis Gasparis, Gilles Godefroy, Bill Johnson, Gin\'{e}s L\'{o}pez, Manuel Maestre and Vicente Montesinos for kindly answering several inquiries.

\end{document}